\documentclass{article}

\usepackage{amsmath,amssymb,amsthm}

\usepackage{cite}

\newtheorem{theorem}{Theorem}
\newtheorem{proposition}[theorem]{Proposition}

\newtheorem{corollary}[theorem]{Corollary}

\newtheorem{definition}[theorem]{Definition}
\newtheorem{remark}[theorem]{Remark}

\newcommand{\R}{\mathbb{R}}

\newcommand{\T}{\mathbb{T}}
\newcommand{\Epi}{\mbox{Epi}}
\newcommand{\Hyp}{\mbox{Hyp}}
\newcommand{\Dom}{\mbox{Dom}}
\newcommand{\Graph}{\mbox{Graph}}
\newcommand{\Inte}{\mbox{Int}}

\begin{document}

\title{\LARGE \bf The contingent epiderivative and 
the calculus of variations on time scales\footnote{Submitted: 06/March/2010; 
revised: 12/May/2010; accepted: 03/July/2010; for publication 
in \emph{Optimization---A Journal of Mathematical Programming and Operations Research}.}}

\author{Ewa Girejko${}^{1, 2}$\\
\texttt{egirejko@ua.pt}
\and Agnieszka B. Malinowska${}^{1, 2}$\\
\texttt{abmalinowska@ua.pt}
\and Delfim F. M. Torres${}^{1, }$\thanks{Corresponding author.}\\
\texttt{delfim@ua.pt}}

\date{${}^1$Department of Mathematics\\
University of Aveiro\\
3810-193 Aveiro, Portugal\\[0.3cm]
${}^2$Faculty of Computer Science\\
Bia{\l}ystok University of Technology\\
15-351 Bia{\l}ystok, Poland}

\maketitle


\begin{abstract}
The calculus of variations on time scales is considered.
We propose a new approach to the subject that consists in applying a differentiation tool
called the \emph{contingent epiderivative}. It is shown that the contingent epiderivative
applied to the calculus of variations on time scales is very useful:
it allows to unify the delta and nabla approaches previously considered in the literature.
Generalized versions of the Euler--Lagrange necessary
optimality conditions are obtained, both for the basic problem
of the calculus of variations and isoperimetric problems.
As particular cases one gets the recent delta and nabla results.

\bigskip


\medskip

\end{abstract}


\section{Introduction}

The \emph{contingent epiderivative} was introduced by J.P.~Aubin (see~\cite{AuFr})
in order to differentiate real valued functions and, in general,
functions with extended values (functions with values in $\R\cup\{\pm\infty\}$).
This derivative belongs to the wide spectrum
of generalized differentiation tools, which are applied
in optimal control theory. For applications of the contingent epiderivative
to optimality conditions one can refer to
\cite{BeSo,JaKh,JaRa,Boris,RoSa} and references therein.
In order to define the contingent epiderivative
one has to be familiar with set-valued maps (maps that have sets as their values):
the definition of the contingent epiderivative is based on the
\emph{contingent derivative}~\cite{AuCe,AuFr} (also called \emph{graphical derivative})
that is a set-valued map that ``differentiates'' set-valued maps.
Namely, we associate with a function $f$ the set-valued map $F_{\uparrow}$ defined
by $F_{\uparrow}(x):=\langle f(x),+\infty)$, whose graph is the epigraph of $f$.
The graphs of the derivatives of such set-valued maps $F_{\uparrow}$ are
the epigraphs of functions which are called \emph{epiderivatives}.
The contingent derivative was also introduced and defined by Aubin
in \cite{AuFr}. In this paper we use such
generalized differentiation tools in order to
unify previous results on the calculus
of variations on time scales.

The mathematics of \emph{time scales} was
introduced by Aulbach and Hilger
as a tool to unify and extend the theories
of difference and differential equations \cite{A:B:O:P:02,A:H,SH}.
A time scale is a model of time,
and the new theory has found important applications in several fields
that require simultaneous modeling of discrete and continuous data,
in particular in the calculus of variations --- see
\cite{A:T,A:B:L:06,B:T:08,B:04,MIA,F:T:07,F:T:08,F:T:09,AM:T,AM:T2,basia:Leitmann,Basia:post_doc_Aveiro:2,NM:T,comNatalia}
and references therein. Looking to the above papers one concludes that
in the calculus of variations on time scales
the problems are formulated and the results are proved in terms
of the delta or the nabla calculus. Here we propose a more
general approach: we make use of contingent epiderivatives
to unify the two different approaches followed in the literature.

Let $\T$ be a given time scale with jump operators $\sigma$ and $\rho$,
and time scale derivatives $\Delta$ and $\nabla$.
We consider as the basic problem of the calculus
of variations to minimize or maximize
\begin{equation*}
 \mathcal{L}[y]
  = \int_a^b L(t, (y\circ \xi_u)(t),D_{\uparrow}\overline{y}(t)(u))d_ut
\end{equation*}
subject to boundary conditions $y(a)=\alpha$ and $y(b)=\beta$,
where $D_{\uparrow}$ denotes the contingent epiderivative,
\begin{equation*}
d_ut:=\\
\left\{
  \begin{array}{ll}
    u\Delta t, \text{ if $u\geq 0$} \\
    u\nabla t, \text{ if $u\leq 0$}\, ,
  \end{array}
\right.
\end{equation*}
and
\begin{equation*}
 y\circ \xi_u:=\\
  \left\{
  \begin{array}{ll}
    u(y\circ \sigma), \text{ if $u\geq 0$} \\
    u(y\circ\rho), \text{ if $u\leq 0$}.
  \end{array}
\right.
\end{equation*}
In the particular case $u = 1$ we obtain the delta variational problem
studied in \cite{B:04}; when $u = -1$ we obtain the nabla
variational problem studied in \cite{A:B:L:06} (see also \cite{NM:T}).
Moreover, we consider isoperimetric problems of the calculus of variations
on time scales: to minimize or maximize $\mathcal{L}[y]$
subject to boundary conditions $y(a)=\alpha$, $y(b)=\beta$,
and an additional constraint of the form
\begin{equation*}
\mathcal{K}[y]=\int_a^b G(t, (y\circ \xi_w)(t),D_{\uparrow}\overline{y}(t)(w))d_wt=K\, .
\end{equation*}
As particular cases we obtain the problem investigated
in \cite{F:T:09} by choosing $u = w= 1$; we obtain
the problem investigated in \cite{A:T} by choosing $u = w = -1$.

Main results of the paper give necessary optimality conditions
of Euler--Lagrange type for the basic and isoperimetric problems
that we propose in this paper. As direct corollaries
we obtain the previous delta and nabla results found in
\cite{A:T,A:B:L:06,B:04,F:T:09}.

The paper is organized as follows. In Section~\ref{sec:prelim} some
preliminary results on time scales are presented.
Our results are then given in Section~\ref{sec:direc deriv}.
Finally, we end with Section~\ref{sec:conc} of main
conclusions and some perspectives on future
directions.


\section{Preliminaries}
\label{sec:prelim}

In this section we introduce the basic definitions and results
that will be needed in the sequel. For a more general presentation
of the theory of time scales and detailed proofs,
we refer the reader to the books
\cite{B:P:01,B:P:03,Lak:book}.


\subsection{The delta calculus}

A {\it time scale} $\mathbb{T}$ is an arbitrary nonempty closed subset
of $\mathbb{R}$.  Besides standard cases of $\mathbb{R}$ (continuous time)
and $\mathbb{Z}$ (discrete time), many different models of time may be used,
\textrm{e.g.}, the $h$-numbers ($\mathbb{T}$ = $h\mathbb{Z}:=\{h z \ | \  z \in
\mathbb{Z}\}$, where $h>0$ is a fixed real number) and the
$q$-numbers ($\mathbb{T}$ = $q^{\mathbb{N}_0}:=\{q^k \ | \  k \in \mathbb{N}_0\}$,
where $q>1$ is a fixed real number). We assume that a time scale
$\mathbb{T}$ has the topology that it inherits from the real numbers
with the standard topology.

For each time scale $\mathbb{T}$ the following operators are used:

\begin{itemize}
\item the {\it forward jump operator} $\sigma:\mathbb{T} \rightarrow \mathbb{T}$,
defined by $\sigma(t):=\inf\{s \in \mathbb{T}:s>t\}$ for $t\neq\sup \mathbb{T}$
and $\sigma(\sup\mathbb{T})=\sup\mathbb{T}$ if $\sup\mathbb{T}<+\infty$;

\item the {\it backward jump operator} $\rho:\mathbb{T} \rightarrow \mathbb{T}$,
defined by $\rho(t):=\sup\{s \in \mathbb{T}:s<t\}$ for $t\neq\inf \mathbb{T}$
and $\rho(\inf\mathbb{T})=\inf\mathbb{T}$ if $\inf\mathbb{T}>-\infty$;

\item the {\it forward graininess function} $\mu:\mathbb{T} \rightarrow [0,\infty)$,
defined by $\mu(t):=\sigma(t)-t$;

\item the \emph{backward graininess function}
$\nu:\mathbb{T}\rightarrow[0,\infty)$, defined by
$\nu(t)=t - \rho(t)$.

\end{itemize}

A point $t\in\mathbb{T}$ is called \emph{right-dense},
\emph{right-scattered}, \emph{left-dense} or \emph{left-scattered}
if $\sigma(t)=t$, $\sigma(t)>t$, $\rho(t)=t$,
and $\rho(t)<t$, respectively. We say that $t$ is \emph{isolated}
if $\rho(t)<t<\sigma(t)$, that $t$ is \emph{dense} if $\rho(t)=t=\sigma(t)$.
If $\sup \mathbb{T}$ is finite and left-scattered, we define
$\mathbb{T}^\kappa := \mathbb{T}\setminus \{\sup\mathbb{T}\}$,
otherwise $\mathbb{T}^\kappa :=\mathbb{T}$.

\begin{definition}
\label{def:de:dif} We say that a function
$f:\mathbb{T}\rightarrow\mathbb{R}$ is \emph{delta differentiable}
at $t\in\mathbb{T}^{\kappa}$ if there exists a number
$f^{\Delta}(t)$ such that for all $\varepsilon>0$ there is a
neighborhood $U$ of $t$ (\textrm{i.e.},
$U=(t-\delta,t+\delta)\cap\mathbb{T}$ for some $\delta>0$) such that
$$|f(\sigma(t))-f(s)-f^{\Delta}(t)(\sigma(t)-s)|
\leq\varepsilon|\sigma(t)-s|,\mbox{ for all $s\in U$}.$$ We call
$f^{\Delta}(t)$ the \emph{delta derivative} of $f$ at $t$ and $f$ is
said to be \emph{delta differentiable} on $\mathbb{T}^{\kappa}$ provided
$f^{\Delta}(t)$ exists for all $t\in\mathbb{T}^{\kappa}$.
\end{definition}

\begin{remark}
If $t \in \mathbb{T} \setminus \mathbb{T}^\kappa$, then
$f^{\Delta}(t)$ is not uniquely defined, since for such a point $t$,
small neighborhoods $U$ of $t$ consist only of $t$ and, besides, we
have $\sigma(t) = t$. For this reason, maximal left-scattered points
are omitted in Definition~\ref{def:de:dif}.
\end{remark}

Note that in right-dense points $f^{\Delta} (t)=\lim_{s\rightarrow
t}\frac{f(t)-f(s)}{t-s}$, provided this limit exists, and in
right-scattered points
$f^{\Delta}(t)=\frac{f(\sigma(t))-f(t)}{\mu(t)}$,
provided $f$ is continuous at $t$.

For $f:\mathbb{T} \rightarrow X$, where $X$ is an arbitrary set,
we define $f^\sigma:=f\circ\sigma$.
For delta differentiable $f$ and $g$, the next formulas hold:
\begin{align*}
f^\sigma(t)&=f(t)+\mu(t)f^\Delta(t) \, ,\\
(fg)^\Delta(t)&=f^\Delta(t)g^\sigma(t)+f(t)g^\Delta(t)\\
&=f^\Delta(t)g(t)+f^\sigma(t)g^\Delta(t).
\end{align*}

In order to describe a class of functions that possess a delta
antiderivative, the following definition is introduced:

\begin{definition}
A function $f:\mathbb{T}\rightarrow\mathbb{R}$ is called \emph{regulated}
provided its right-sided limits exists (finite) at all right-dense points
in $\T$ and its left-sided limits exists (finite)
at all left-dense points in $\T$.
\end{definition}

\begin{definition}
A function $f:\mathbb{T}\rightarrow\mathbb{R}$ is called
\emph{rd-continuous} if it is continuous at right-dense points
in $\mathbb{T}$ and its left-sided limits exist (finite)
at left-dense points in $\mathbb{T}$.
\end{definition}

\begin{definition}
A function $F:\mathbb{T}\rightarrow\mathbb{R}$ is called a
\emph{delta antiderivative} of
$f:\mathbb{T}\rightarrow\mathbb{R}$ provided
$F^{\Delta}(t)=f(t)$ for all $t \in \mathbb{T}^\kappa$.
In this case we define the \emph{delta integral} of $f$ from $a$
to $b$ ($a,b \in \mathbb{T}$) by
\begin{equation*}
\int_{a}^{b}f(t)\Delta t:=F(b)-F(a) \, .
\end{equation*}
\end{definition}

\smallskip

\begin{theorem}(Theorem~1.74 in~\cite{B:P:01})
Every rd-continuous function has a delta
antiderivative. In particular, if $a \in \mathbb{T}$,
then the function $F$ defined by
$$
F(t)= \int_{a}^{t}f(\tau)\Delta\tau, \quad t \in \mathbb{T} \, ,
$$
is a delta antiderivative of $f$.
\end{theorem}


\subsection{The nabla approach}

In order to introduce the definition of nabla derivative, we define
a new set $\mathbb{T}_\kappa$ which is derived from $\mathbb{T}$ as
follows: if  $\mathbb{T}$ has a right-scattered minimum $m$, then
$\mathbb{T}_\kappa=\mathbb{T}\setminus\{m\}$; otherwise,
$\mathbb{T}_\kappa= \mathbb{T}$. In order to simplify expressions,
and similarly as done with composition with $\sigma$, we define
$f^{\rho}(t) := f(\rho(t))$.
We say that a function $f:\mathbb{T}\rightarrow\mathbb{R}$ is
\emph{nabla differentiable} at $t\in\mathbb{T}_\kappa$ if there is a
number $f^{\nabla}(t)$ such that for all $\varepsilon>0$ there
exists a neighborhood $U$ of $t$ (\textrm{i.e.},
$U=(t-\delta,t+\delta)\cap\mathbb{T}$ for some $\delta>0$) such that
$$|f^\rho(t)-f(s)-f^{\nabla}(t)(\rho(t)-s)|
\leq\varepsilon|\rho(t)-s|,\mbox{ for all $s\in U$}.$$ We call
$f^{\nabla}(t)$ the \emph{nabla derivative} of $f$ at $t$. Moreover,
we say that $f$ is \emph{nabla differentiable} on $\mathbb{T}$
provided $f^{\nabla}(t)$ exists for all $t \in \mathbb{T}_\kappa$.
If $f$ is nabla differentiable at $t$, then
$$f^\rho(t)=f(t)-\nu(t)f^\nabla(t) \, .$$

\begin{theorem}(Theorem~8.41 in \cite{B:P:01})
Suppose $f,g:\mathbb{T}\rightarrow\mathbb{R}$ are nabla
differentiable at $t\in\mathbb{T}_\kappa$. Then,
\begin{enumerate}
\item the sum $f+g:\mathbb{T}\rightarrow\mathbb{R}$ is nabla
differentiable at $t$ and
$$(f+g)^{\nabla}(t)=f^{\nabla}(t) + g^{\nabla}(t) \, ;$$

\item for any constant $\alpha$, $\alpha
f:\mathbb{T}\rightarrow\mathbb{R}$ is nabla differentiable at $t$
and $$(\alpha f)^{\nabla} (t)=\alpha f^{\nabla}(t)\, ;$$

\item the product $fg:\mathbb{T}\rightarrow\mathbb{R}$ is
    nabla differentiable at $t$ and
$$
\begin{array}{rcl}
(fg)^{\nabla}(t)& = & f^{\nabla}(t)g(t) + f^{\rho}(t)g^{\nabla}(t)\\
&=& f^{\nabla}(t)g^{\rho}(t)+ f(t)g^{\nabla}(t).
\end{array}
$$
\end{enumerate}
\end{theorem}

\begin{definition}
A function $F:\mathbb{T}\rightarrow\mathbb{R}$ is called a
\emph{nabla antiderivative} of $f:\mathbb{T}\rightarrow\mathbb{R}$
provided $F^{\nabla}(t)=f(t)$ for all $t \in \mathbb{T}_\kappa$.
In this case we define the \emph{nabla integral} of $f$ from $a$ to
$b$ ($a,b \in \mathbb{T}$) by
$$
\int_{a}^{b}f(t)\nabla t:=F(b)-F(a) \, .
$$
\end{definition}

In order to exhibit a class of functions that possess a nabla
antiderivative, the following definition is introduced.

\begin{definition}
Let $\mathbb{T}$ be a time scale,
$f:\mathbb{T}\rightarrow\mathbb{R}$. We say that function $f$ is
\emph{ld-continuous} if it is continuous at left-dense points
and its right-sided limits exist (finite) at all right-dense points.
\end{definition}

\begin{theorem}(Theorem~8.45 in \cite{B:P:01})
Every ld-continuous function has a nabla
antiderivative. In particular, if $a \in \mathbb{T}$, then the
function $F$ defined by
$$
F(t)= \int_{a}^{t}f(\tau)\nabla\tau, \quad t \in \mathbb{T} \, ,
$$
is a nabla antiderivative of $f$.
\end{theorem}

For more on the nabla calculus we refer the reader to
\cite[Chap.~3]{B:P:03}.


\section{Main Results}
\label{sec:direc deriv}

In this section, which is the main part of the paper,
we present a unified Euler--Lagrange equation
and unified necessary optimality conditions for
isoperimetric problems on time scales.
The differentiation tool we use for such unification
is the contingent epiderivative.
We consider $\T$ to be a given time scale with
$\inf\T:=a$, $\sup\T:=b$, and $I:=[a,b]\cap\T$ for $[a,b]\subset\R$.
Moreover, by $I^{\kappa}_{\kappa}$ (or $\T^{\kappa}_{\kappa}$)
we mean $I^{\kappa}_{\kappa}:=I^{\kappa}\cap I_{\kappa}$
(respectively, $\T^{\kappa}_{\kappa}:=\T^{\kappa}\cap \T_{\kappa}$).


\subsection{Contingent epiderivatives}
\label{sec:4}

A \emph{set-valued map} (\emph{multifunction}) $F:\R\rightarrow2^{\R}$
is a map that has sets as its values. We denote it by $F:\R\twoheadrightarrow\R$.
Every set-valued map $F:\R\twoheadrightarrow\R$ is characterized
by its \emph{graph}, $\Graph(F)$, as the subset of $\R^2$ defined by
\begin{equation*}
\Graph(F):=\{(x,y)\in\R^2: y\in F(x)\}\,.
\end{equation*}
We shall say that $F(x)$ is the \emph{image} or the \emph{value} of $F$ at $x$.
A set-valued map is said to be \emph{nontrivial} if its graph is not empty,
\textrm{i.e.}, if there exists an element $x\in\R$
such that $F(x)$ is not empty. The \emph{domain} of $F$
is the set of elements $x\in\R$ such that $F(x)$ is not empty:
$\Dom(F)=\{x\in\R: F(x)\neq \emptyset\}$.
The \emph{image} of $F$ is the union of the images (or values) $F(x)$,
when $x$ ranges over $\R$: $\mbox{Im}(F)=\bigcup_{x\in\R}F(x)$.

\begin{definition} \label{def:epi} Let $X$ be any nonempty subset of $\R$.
By \emph{the epigraph of} $f:X\rightarrow\R$, denoted by $\Epi(f)$,  we mean the following set:
\begin{equation*}
\Epi(f):=\{(t,\lambda)\in X\times \R : f(t)\leq \lambda \}\, .
\end{equation*}
\emph{The hypograph of} $g:X\rightarrow \R$ is defined in the symmetric way:
$\Hyp(g):=\{(t,\lambda)\in X\times \R : g(t)\geq \lambda \}$.
\end{definition}

\begin{remark}
Note that $\Epi(f)\cap\Hyp(f)=\Graph(f)$.
\end{remark}

If $X=\T$ is a time scale, then we can use the same definition
of epigraph by introducing an appropriate extension of the epigraph
of a function  $f:\T\rightarrow\R$.
Let $G(f)$ denote the following set:
\begin{equation*}
G(f)=\bigcup_{t\in \T}\left\{\alpha(t,y)+\beta(\sigma(t),z):
\, y\geq f(t), z\geq f^{\sigma}(t)\right\}\, ,
\end{equation*}
where $\alpha+\beta=1, \alpha, \beta \geq 0$.

\begin{remark}
It is easy to see that $G(f)\subset\overline{conv} \Epi(f)$,
where $\overline{conv} \Epi(f)$ is the closure of the convex hull of $\Epi(f)$.
\end{remark}

Using  the formulation of $G(f)$  we can assign to $f:I\rightarrow\R$
the function $\overline{f}: [a,b]\rightarrow \R$ defined by the formula
\begin{equation}
\label{eq:epi}
\Epi(\overline{f})=G(f).
\end{equation}
Let us notice that for $f,g:I\rightarrow \R$
and $a,b\in\R$ it holds:
$a\overline{f}+b\overline{g}=\overline{af+bg}$.

\begin{remark}\label{rem:epi}
Note that function $\overline{f}$ defined by formula \eqref{eq:epi} can be presented
in the following way (see, \textrm{e.g.}, \cite{Di}):
\[\overline{f}(t)=
\begin{cases}
f(t), &\text{if $t\in \T$}\\
f(s)+\frac{f(\sigma(s))-f(s)}{\mu(s)}(t-s), &\text{if $t\in(s,\sigma(s)),$ }
\end{cases}\]
 where $s\in \T$ and $s$ is right-scattered or
\[\overline{f}(t)=
\begin{cases}
f(t), &\text{if $t\in \T$}\\
f(s)+\frac{f(s)-f(\rho(s))}{\nu(s)}(t-s), &\text{if $t\in(\rho(s),s),$}
\end{cases}\]
where $s\in \T$ and $s$ is left-scattered.
\end{remark}

We recall a general definition of the contingent cone to a subset of $\R^2$.
This theory is presented for a subset of any normed space $X$ in \cite{AuFr}.

\begin{definition}[\cite{AuFr}]
Let $K\subset \R^2$ and  point  $p\in\overline{K}$ belong to the closure of $K$.
The contingent cone $T_K(p)$ is defined by
\begin{equation*}
T_K(p):=\{v: \liminf _{h\rightarrow 0+} d(p+hv,K)/h=0\}\,,
\end{equation*}
where $d(q,K)=\inf\limits_{k\in K}d(q,k)$ is the distance between
point $q=p+hv$ and set $K$.
\end{definition}

\begin{remark}[\cite{AuFr}]
$T_K(p)$ is a closed set.
\end{remark}
The following characterization
of the contingent cone in terms of sequences is very convenient:
$v\in T_K(p)$ if and only if
there exists $h_n\rightarrow 0^+$ and
$v_n\rightarrow v$ such that for all
$n$, $p+h_nv_n\in K$.

\begin{remark}[\cite{AuFr}]
If $p\in \Inte(K)$, then $T_K(p)=\R^2$.
\end{remark}

\begin{definition}[\cite{AuFr}]
Let $F:\R\twoheadrightarrow\R$ be a set-valued map.
The \emph{contingent derivative of $F$ at $p\in\Graph(F)$}, denoted by  $DF(p)$,
is the set-valued map from $\R$ to $\R$ defined by
\[\Graph\left(DF(p)\right):=T_{\Graph(F)}(p)\,. \]
For $F:=f$ a single-valued function, we set
$$Df(x):=Df(x,f(x)).$$
\end{definition}

Let us  point  out  that
\begin{equation*}
\Graph\left(DF(p)\right)=\{u=(u_1,u_2): u_2\in DF(p)(u_1)\}
\end{equation*}
and $\ u\in T_{Graph(F)}(p)$.
In \cite{AuFr} one can find the following characterization
of the contingent derivative:

\begin{proposition}[\textrm{cf.} \cite{AuFr}, Proposition 5.1.4, p.186]
Let $F:\R\twoheadrightarrow\R$ be a set-valued
map and let $p=(x,y)\in \Graph(F)$. Then,
\[v\in DF(p)(u)\Leftrightarrow \liminf_{h\rightarrow 0+,
u'\rightarrow u} d\left(v, \frac{F(x+hu')-y}{h}\right)=0\,.\]
If $p\in\mbox{Int}\left(\Dom(F)\right)$ and $F$ is Lipschitz around $p$, then
\[v\in DF(p)(u)\Leftrightarrow \liminf_{h\rightarrow 0+}
d\left(v, \frac{F(x+hu)-y}{h}\right)=0\,.\]
\end{proposition}
Let $X\subset\R$, possibly a time scale, and let us consider
an extended function $f: X\rightarrow \R\cup\{\pm\infty\}$
whose domain is $\Dom (f)=\{t\in X: f(t)\neq \pm \infty\}$.
We call an extended function \emph{nontrivial} if $\Dom (f)\neq\emptyset$.
To introduce elements of the theory of contingent epiderivatives
for functions on time scales we define two set-valued maps
$F_{\uparrow}: X\twoheadrightarrow\R$
and $F_{\downarrow}: X\twoheadrightarrow\R$, both corresponding
to $f:X\rightarrow\R$,  in the following way (see \cite{AuFr}):

\begin{definition}\label{def:mult_f}
Let $X\subset \R$ and $f: X\rightarrow \R\cup\{\pm \infty\}$
be an extended function  with nonempty domain $\Dom (f)\neq\emptyset$. Then \\
a) $F_{\uparrow}(t)=\left\{\begin{array}{ll}f(t)+\R_+
&  \mbox{if} \ t\in \Dom(f)\\
\emptyset  & \mbox{if} \ f(t)=+\infty\\
\R & \mbox{if} \ f(t)=-\infty
\end{array}\right.$, \\
b) $F_{\downarrow}(t)=\left\{\begin{array}{ll}f(t)-\R_+
&  \mbox{if} \ t\in \Dom(f)\\
\emptyset  & \mbox{if} \ f(t)=-\infty\\
\R & \mbox{if} \ f(t)=+\infty
\end{array}\right.$,\\
where $f(t)+\R_{+}=\{\lambda: f(t)\leq \lambda\}$
and $f(t)-\R_{+}=\{\lambda: f(t)\geq \lambda\}$.

\end{definition}
\begin{remark}
It is easy to prove that $\Graph(F_{\uparrow})=\Epi(f)
=\bigcup_t\{(t,y): y\in F_{\uparrow}(t) \}$
and $\Graph(F_{\downarrow})=\Hyp(f)$.
\end{remark}
With set-valued functions $F_{\uparrow}$
and $F_{\downarrow}$ one can associate their contingent derivatives.
We naturally have that values of the contingent derivative
of $F_{\uparrow}: \R\twoheadrightarrow\R$
are half lines in  the sense that  $\forall$ $\lambda\geq f(t)$,
$\forall$ $u\in \Dom\left(DF_{\uparrow}(t,\lambda)\right)$,
it holds:
\[DF_{\uparrow}(t,\lambda)(u)=DF_{\uparrow}(t,\lambda)(u)+\R_+ \,.\]

\begin{definition}[\cite{AuFr}]
\label{def:contepider}
Let $f:X\rightarrow\R\cup\{\pm\infty\}$ be a function
with extended values and nonempty domain and let $t\in \Dom(f)$.
We say that the extended function $D_{\uparrow}f(t):\R\rightarrow\R\cup\{\pm\infty\}$ defined by
\[ D_{\uparrow}f(t)(u):=\inf\{v: v\in DF_{\uparrow}(t, f(t))(u)\}  \, ,\]
where $F_\uparrow$ is given by Definition~\ref{def:mult_f} and $u\in\R$,
is the \emph{contingent epiderivative of $f$ at $t$ in the direction $u$}.
The function $f$ is said to be \emph{contingently epidifferentiable at $t$} if
its contingent epiderivative never takes the value $-\infty$.
If  $DF_{\uparrow}(t, f(t))(u)=\emptyset$ then we set  $D_{\uparrow}f(t)(u)=+\infty$.
\end{definition}

\begin{remark}[\cite{AuFr}]
Notice that for a function $f:X\rightarrow\R\cup\{\pm\infty\}$
with nonempty domain the following conditions are equivalent:
 \begin{itemize}
   \item[(i)] $f$ is contingently epidifferentiable at $t\in \Dom (f)$;
   \item[(ii)] $D_{\uparrow}f(t)(0)=0$.
 \end{itemize}
 \end{remark}

In \cite{AuCe} one can find the following characterization
of the contingent epiderivative for some particular situations:

a) $DF_{\uparrow}(t, f(t))(u)=\R$ iff $D_{\uparrow}f(t)(u)=-\infty$;

b) $DF_{\uparrow}(t, f(t))(u)=[v_0,+\infty)$ iff $D_{\uparrow}f(t)(u)=v_0$.

Now let associate with $f:I\rightarrow\R$ the function
$\overline{f}:[a,b]\rightarrow\R\cup\{\pm\infty\}$ defined by formula \eqref{eq:epi}.
Let $\overline{F}_{\uparrow}$ denote the corresponding multifunction
for $\overline{f}$ according to Definition~\ref{def:mult_f}.

\begin{proposition}
\label{prop:2}
For all $p\in \Graph\left(F_{\uparrow}\right)$ and for each $u\in\R$  we have
\begin{equation}
\label{eq:prop12}
DF_{\uparrow}(p)(u)\subset  D\overline{F}_{\uparrow}(p)(u)\,.
\end{equation}
\end{proposition}

\begin{proof}
Since $\Graph(F_{\uparrow})\subset \Graph(\overline{F}_{\uparrow})$
and $\Graph\left(DF_{\uparrow}(p)\right)=T_{\Graph(F_{\uparrow})}(p)$, it follows that
$\Graph(DF_{\uparrow}(p))\subset \Graph(D\overline{F}_{\uparrow}(p))$.
Let $v\in DF_{\uparrow}(p)(u)$. Then $(u,v)\in \Graph\left(DF_{\uparrow}(p)\right)
\subset \Graph\left(D\overline{F}_{\uparrow}(p)\right)$ and we get that $v\in D\overline{F}_{\uparrow}(p)(u)$.
\end{proof}

\begin{remark}
Inclusion (\ref{eq:prop12}) is also satisfied when both sets are empty.
It means then that the contingent epiderivative $D_{\uparrow}f(t)(u)=+\infty$.
\end{remark}

\begin{remark}
Directly from definitions of $\overline{F}_{\uparrow}$
and its contingent derivative we have the following:
\begin{enumerate}
\item  $DF_{\uparrow}(t)(u)=D\overline{F}_{\uparrow}(t)(u)$
if $t\in\T$ is right-dense and $u\geq 0$;

\item  $DF_{\uparrow}(t)(u)=D\overline{F}_{\uparrow}(t)(u)$
if $t\in\T$ is left-dense and $u\leq 0$;

\item  $D\overline{F}_{\uparrow}(t)(u)=\emptyset$
if $t\in\T$ is isolated and $u\neq 0$.
\end{enumerate}
\end{remark}

We recall the following useful proposition from \cite{AuFr} showing that
the contingent epiderivative can be characterized as a limit of differential quotients:

\begin{proposition}[\cite{AuFr}]
\label{prop:AuFr}
Let $g:I\rightarrow\R\cup\{\pm\infty\}$ be a nontrivial extended function and $t$ belong to its domain. Then
\begin{equation*}
D_{\uparrow}g(t)(u)=\liminf_{h\rightarrow0+,u'\rightarrow u}\frac{g(t+hu')-g(t)}{h}.
\end{equation*}
\end{proposition}

\begin{corollary}
Let $f:\T\rightarrow\R\cup\{\pm\infty\}$ and $t\in \Dom(f)$. Then for $u\in\R$
\begin{equation*}
D_{\uparrow}f(t)(u)\geq \liminf_{h\rightarrow 0+, u'\rightarrow u}\frac{\overline{f}(t+hu')
-\overline{f}(t)}{h}=D_{\uparrow}\overline{f}(t)(u)\,.
\end{equation*}
\end{corollary}

\begin{proof}
From Proposition~\ref{prop:2} and definition~\ref{def:contepider}
we have $D_{\uparrow}f(t)(u)\geq D_{\uparrow}\overline{f}(t)(u)$.
The equality follows from Proposition~\ref{prop:AuFr}.
\end{proof}

We can state the following relations between delta
and nabla derivatives of $f$ at point $t$ (if  they exist)
and the  contingent epiderivative of the corresponding
function $\overline{f}$ defined by formula~\eqref{eq:epi}.

\begin{proposition}[\cite{G:M:W:09}]
\label{prop:Dupf+deriv}
Let $t\in\T^{\kappa}_{\kappa}$ and $f:I\rightarrow\R\cup\{\pm\infty\}$.

a) If  $f^{\Delta}(t)$ exists, then
$D_{\uparrow}\overline{f}(t)(u)=uf^{\Delta}(t)$
for $u\geq 0$.

b)  If  $f^{\nabla}(t)$ exists, then
$D_{\uparrow}\overline{f}(t)(u)=uf^{\nabla}(t)$
for $u\leq0$.
\end{proposition}

We can characterize  values of the contingent epiderivative
of $f:\T\rightarrow\R\cup\{\pm\infty\}$
by the contingent epiderivative of $\overline{f}$. Directly
from definitions of $\overline{f}$ and its contingent epiderivative
we have the following:

\begin{proposition}[\cite{G:M:W:09}]
Let $f:I\rightarrow\R\cup\{\pm\infty\}$, $\overline{f}:[a,b]\rightarrow\R\cup\{\pm\infty\}$,
and $Epi(\overline{f})=G(f)$. Then,
\begin{enumerate}
\item $f$ is contingently epidifferentiable
if and only if $\overline{f}$ is contingently epidifferentiable;

\item $D_{\uparrow}f(t)(u)=D_{\uparrow}\overline{f}(t)(u)$
for $t\in\T$ right-dense and $u\geq 0$;

\item $D_{\uparrow}f(t)(u)=D_{\uparrow}\overline{f}(t)(u)$
for $t\in\T$ left-dense and $u\leq 0$;

\item $D_{\uparrow}\overline{f}(t)(u)=\R_+$
for $t\in\T$ isolated and $u\neq 0$.
\end{enumerate}
\end{proposition}


We introduce now the following notations and definitions:

\begin{definition}\label{def:uni}
Let $u\in\R$ be a real number. Then,
\begin{equation*}
d_ut:=\\
\left\{
  \begin{array}{ll}
    u\Delta t, \text{ if $u\geq 0$} \\
    u\nabla t, \text{ if $u\leq 0$}\, ,
  \end{array}
\right.
\end{equation*}

\begin{equation*}
 y\circ \xi_u:=\\
  \left\{
  \begin{array}{ll}
    u(y\circ \sigma), \text{ if $u\geq 0$} \\
    u(y\circ\rho), \text{ if $u\leq 0$} \, .
  \end{array}
\right.
\end{equation*}
\end{definition}

\begin{remark}
With the notation of Definition~\ref{def:uni} we have
\begin{equation*}
\int_a^b \overline{f}(t)d_ut:=\\
\left\{
  \begin{array}{ll}
    u\int_a^b f(t)\Delta t, \text{ if $u\geq 0$} \\
    u\int_a^b f(t)\nabla t, \text{ if $u\leq 0$} \, ,
  \end{array}
\right.
\end{equation*}
where $\overline{f}$ is defined by formula~\eqref{eq:epi}.
\end{remark}


\subsection{The unified Euler--Lagrange equation}
\label{Uni E-L}

Given a time scale $\T$ and $u\in \R \setminus \{0\}$,
we consider the question of finding $y$ that is a solution to the problem

\begin{equation}
\label{problem:Pu}
\begin{gathered}
\text{extremize}\ \ \mathcal{L}[y] =
\int_a^b L(t,(y\circ \xi_u)(t),D_{\uparrow}\overline{y}(t)(u))d_ut\,,\\
y(a)=\alpha,\quad y(b)=\beta \, ,
\end{gathered}
\end{equation}
where function $(t,y,v)\rightarrow L(t,y,v)$ from $I\times\R^2$ to
$\R$ has partial continuous derivatives with respect to $y,v$ for all $t\in I$.

Let $\partial_{i}L$ denote the standard
partial derivative of $L(\cdot,\cdot,\cdot)$
with respect to its $i$th variable, $i = 1,2,3$.

Any regulated function $y:I\rightarrow\R$ such that
$D_{\uparrow}\overline{y}(t)(u)\neq\{-\infty,+\infty\}$
is said to be an \emph{admissible} function provided
that it satisfies boundary conditions $y(a)=\alpha$, $y(b)=\beta$.

\begin{remark}
Proposition~\ref{prop:Dupf+deriv} implies the following:
if $y$ is $\Delta$-differentiable, then for $u=1$
\eqref{problem:Pu} is just a delta problem of the calculus of variations
on time scales (see \cite{B:T:08,B:04}); while if $y$
is $\nabla$-differentiable, then for $u=-1$ \eqref{problem:Pu} reduces
to a nabla problem of the calculus of variations
(see \cite{A:B:L:06,NM:T}).
For $u=0$ the cost-functional to extremize reduces to a constant,
so problem \eqref{problem:Pu} is trivial: there is nothing to minimize or maximize;
any function satisfying $y(a)=\alpha$ and $y(b)=\beta$ is a solution to \eqref{problem:Pu}.
Thus, we are interested in the nontrivial case $u \ne 0$.
\end{remark}

\begin{theorem}{\rm (The unified Euler--Lagrange equation
on time scales).}
\label{thm:uni EL}
If $y$ is a minimizer or maximizer to problem
\eqref{problem:Pu}, then $y$ satisfies
the following equation:
\begin{equation}
\label{uni EL}
D_{\uparrow}\left(\partial_3 L\left(t, (y\circ \xi_u)(t),D_{\uparrow}\overline{y}(t)(u)\right)\right)(u)
=u \cdot\partial_2 L(t, (y\circ \xi_u)(t),D_{\uparrow}\overline{y}(t)(u))
\end{equation}
for all $t\in I^{\kappa^2}_{\kappa^2}$,
where $\overline{y}$ is defined by formula~\eqref{eq:epi}.
\end{theorem}

\begin{proof}
We consider two cases: $u>0$ and $u<0$.
For $u>0$ problem~\eqref{problem:Pu}
reduces to
\begin{equation}
\label{eq:uni:delta}
\begin{gathered}
\text{extremize}\ \ \int_a^b uL(t, u(y\circ \sigma)(t),uy^{\Delta}(t))\Delta t,\\
y(a)=\alpha,\ y(b)=\beta.
\end{gathered}
\end{equation}
If we set $f(t,y^{\sigma}(t),y^{\Delta}(t))
:= uL(t, u y^{\sigma}(t),uy^{\Delta}(t))$,
then problem~\eqref{eq:uni:delta} is equivalent to
\begin{equation}
\label{eq:int}
\begin{gathered}
\text{extremize}\ \ \int_a^b f(t,y^{\sigma}(t),y^{\Delta}(t))\Delta t\\
 y(a)=\alpha,\ y(b)=\beta.
\end{gathered}
\end{equation}
But this is exactly the case of a standard delta problem of the calculus
of variations on time scales. For problem~\eqref{eq:int}
the Euler--Lagrange equation of \cite{B:04} gives the delta equation
\begin{equation*}
\left[\partial_3f(t,y^{\sigma}(t),y^{\Delta}(t))\right]^{\Delta}
=\partial_2f(t,y^{\sigma}(t),y^{\Delta}(t)),
\end{equation*}
which is equivalent to
\begin{equation}\label{eq:EL:u}
u\left[\partial_3L(t, u(y\circ \sigma)(t),uy^{\Delta}(t))\right]^{\Delta}
=u\partial_2L(t, u(y\circ \sigma)(t),uy^{\Delta}(t)) \, .
\end{equation}
By assumption $u>0$ formula \eqref{eq:EL:u} is equivalent to
\begin{equation*}
\left[\partial_3L(t, u(y\circ \sigma)(t),uy^{\Delta}(t))\right]^{\Delta}
=\partial_2L(t, u(y\circ \sigma)(t),uy^{\Delta}(t)),
\end{equation*}
and we obtain \eqref{uni EL} for $u>0$.\\
Similarly, let us take $u<0$. Then problem~\eqref{problem:Pu}
reduces to the following nabla problem of the calculus of variations:
\begin{equation}
\label{eq:uni:nabla}
\begin{gathered}
\text{extremize}\ \ \int_a^b uL(t, u(y\circ \rho)(t),uy^{\nabla}(t))\nabla t,\\
y(a)=\alpha,\ y(b)=\beta.
\end{gathered}
\end{equation}
If we set $g(t,y^{\rho}(t),y^{\nabla}(t))
:= u L(t, u y^{\rho}(t),uy^{\nabla}(t))$,
then problem~\eqref{eq:uni:nabla} is equivalent to
\begin{gather*}
\text{extremize}\ \ \int_a^b g(t,y^{\rho}(t),y^{\nabla}(t))\nabla t,\\
 y(a)=\alpha,\ y(b)=\beta.
\end{gather*}
This is the case of a standard nabla problem of the calculus of variations
on time scales \cite{A:B:L:06,NM:T}, and we get as necessary
optimality condition the nabla differential equation
\begin{equation*}
\left[\partial_3g(t,y^{\rho}(t),y^{\nabla}(t))\right]^{\nabla}=\partial_2g(t,y^{\nabla}(t),y^{\nabla}(t))
\end{equation*}
that one can write equivalently as
\begin{equation*}
\left[\partial_3L(t, u(y\circ \rho)(t),uy^{\nabla}(t))\right]^{\nabla}\\
=\partial_2L(t, u(y\circ \rho)(t),uy^{\nabla}(t)) \, ,
\end{equation*}
\textrm{i.e.}, we obtain \eqref{uni EL} for $u<0$.
\end{proof}


\subsection{The general isoperimetric problem}
\label{Uni Iso}

Let $u,w\in\R \setminus \{0\}$
and consider the general isoperimetric problem on time scales.
The problem consists of extremizing
\begin{equation}
\label{problem:P:uni}
\mathcal{L}[y]=\int_a^b L(t, (y\circ \xi_u)(t),D_{\uparrow}\overline{y}(t)(u))d_ut
\end{equation}
subject to the boundary conditions
\begin{equation}
\label{bou:con:uni}
y(a) = \alpha \, , \quad y(b) = \beta \, ,
\end{equation}
and the constraint
\begin{equation}
\label{const:uni}
\mathcal{K}[y]=\int_a^b G(t, (y\circ \xi_w)(t),D_{\uparrow}\overline{y}(t)(w))d_wt=K \, ,
\end{equation}
where $\alpha$, $\beta$, and $K$ are given real numbers.

\begin{definition}
We say that ${y}$ is an \emph{extremal}
for $\mathcal{K}$ if ${y}$ satisfies the equation
\begin{multline*}
D_{\uparrow}\left(\partial_3 G(t, (y\circ \xi_w)(t),D_{\uparrow}\overline{y}(t)(w))\right)(w)\\
=w\partial_2 G(t, (y\circ \xi_w)(t),D_{\uparrow}\overline{y}(t)(w)),
\ \ \forall t\in I^{\kappa^2}_{\kappa^2}\, ,
\end{multline*}
where $\overline{y}$ is defined by formula~\eqref{eq:epi}.
An extremizer (\textrm{i.e.}, a minimizer or a maximizer)
for problem \eqref{problem:P:uni}--\eqref{const:uni} that is
not an extremal for $\mathcal{K}$ is said to be a normal extremizer;
otherwise (\textrm{i.e.}, if it is an extremal for $\mathcal{K}$), the
extremizer is said to be abnormal.
\end{definition}

Proofs of the following two theorems are done
following the techniques in the proof of Theorem~\ref{thm:uni EL}
by considering three cases: (i) $u$ and $w$ are both positive;
(ii) $u$ and $w$ are both negative;
(iii) $u$ and $w$ are of different signs.
Indeed, in case (i) the isoperimetric problem
\eqref{problem:P:uni}--\eqref{const:uni}
is reduced to the one studied in \cite{F:T:09};
in case (ii) we can apply the results in \cite{A:T}
to obtain Theorem~\ref{thm:mr:uni} and Theorem~\ref{thm:mr:uni:abn}
in the case $u, w < 0$; when $sign(u w) = -1$ we need to use
the necessary optimality conditions for the delta-nabla isoperimetric
problems investigated in \cite{China}.

\begin{theorem}
\label{thm:mr:uni}
If ${y}$ is a normal extremizer for the isoperimetric problem
\eqref{problem:P:uni}--\eqref{const:uni}, then ${y}$ satisfies the
following equation for all $t\in I^{\kappa^2}_{\kappa^2}$:
\begin{multline*}
 D_{\uparrow}(\partial_3 L(t, (y\circ \xi_u)(t),D_{\uparrow}\overline{y}(t)(u)))(u)
- u\partial_2 L(t, (y\circ \xi_u)(t),D_{\uparrow}\overline{y}(t)(u))\\
=\lambda \Bigl( D_{\uparrow}(\partial_3 G(t, (y\circ \xi_w)(t),D_{\uparrow}\overline{y}(t)(w)))(w)
- w\partial_2 G(t, (y\circ \xi_w)(t),D_{\uparrow}\overline{y}(t)(w))\Bigr).
\end{multline*}
\end{theorem}

Next theorem introduces an extra multiplier $\lambda_0$,
allowing to cover abnormal extremizers.

\begin{theorem}
\label{thm:mr:uni:abn}
If ${y}$ is an extremizer for the isoperimetric problem \eqref{problem:P:uni}--\eqref{const:uni},
then there exist two constants $\lambda_0$ and $\lambda$, not both equal to zero,
such that ${y}$ satisfies the following equation for all $t\in I^{\kappa^2}_{\kappa^2}$:
\begin{multline*}
\lambda_0 \Bigl(D_{\uparrow}(\partial_3 L(t, (y\circ \xi_u)(t),D_{\uparrow}\overline{y}(t)(u)))(u)
-u\partial_2 L(t, (y\circ \xi_u)(t),D_{\uparrow}\overline{y}(t)(u))\Bigr)\\
=\lambda \Bigl(D_{\uparrow}(\partial_3 G(t, (y\circ \xi_w)(t),D_{\uparrow}\overline{y}(t)(w)))(w)
-w\partial_2 G(t, (y\circ \xi_w)(t),D_{\uparrow}\overline{y}(t)(w))\Bigr) .
\end{multline*}
\end{theorem}


\section{Conclusions}
\label{sec:conc}

In this paper we claim that the contingent epiderivative,
a concept introduced by Aubin with the help
of his contingent derivative, is an important tool in the
theory of time scales. To show this we propose a new
approach to the calculus of variations on time scales, which allows to unify the
three different approaches followed so far in the literature:
the delta \cite{B:04,F:T:09,AM:T}; the nabla \cite{A:T,A:B:L:06};
and the delta-nabla \cite{China,AM:T:Poland} approach.
Such unification is obtained using, for the first time in the literature of time scales,
the contingent epiderivative as the main differentiation tool.
The effectiveness of our approach is illustrated
by giving general necessary optimality conditions
both for the basic and isoperimetric problems of the calculus
of variations on time scales.
The generalized versions of the Euler-Lagrange
conditions are reduced, in particular cases,
to recent results obtained in the delta and nabla frameworks.

We trust that the present work opens a new area
of research. In particular, it would be interesting
to generalize our results for problems of the calculus
of variations on time scales with higher-order contingent epiderivatives,
unifying the results of \cite{F:T:08} and \cite{NM:T}.
In the case of optimal control problems governed by ODE systems
(controlled ODEs or differential inclusions) and discrete-time systems
there are necessary optimality conditions of different types:
mainly using derivative-type constructions (in particular, epi-derivatives)
and the coderivative type for the Euler-Lagrange inclusions.
These are some interesting questions needing further developments.


\section*{Acknowledgments}

The authors are supported by the R\&D unit
\emph{Centre for Research on Optimization and Control}
(CEOC) via \emph{The Portuguese Foundation for
Science and Technology} (FCT) and the European Community
fund FEDER/POCI 2010. Girejko is also supported by the FCT
post-doc fellowship SFRH/BPD/48439/2008;
Malinowska by Bia{\l}ystok University of Technology,
via a project of the Polish Ministry of Science
and Higher Education \emph{Wsparcie miedzynarodowej mobilnosci naukowcow};
and Torres by the project UTAustin/MAT/0057/2008.




\begin{thebibliography}{99}

\bibitem{A:B:O:P:02}
R. Agarwal, M. Bohner, D. O'Regan and A. Peterson,
Dynamic equations on time scales: a survey.
Dynamic equations on time scales,
{\it J. Comput. Appl. Math.}, vol.~141, no.~1-2, 2002, pp~1--26.

\bibitem{A:T}
R. Almeida and D.F.M. Torres,
Isoperimetric problems on time scales with nabla derivatives,
{\it J. Vib. Control}, vol.~15, no.~6, 2009, pp~951--958.
{\tt arXiv:0811.3650}

\bibitem{A:B:L:06}
F.M. Atici, D.C. Biles and A. Lebedinsky,
An application of time scales to economics,
{\it Math. Comput. Modelling}, vol.~43, no.~7-8,
2006, pp~718--726.

\bibitem{AuCe}
J.P. Aubin and A. Cellina,
{\it Differential Inclusions}, Springer-Verlag, Berlin,
Heidelberg, New York and Tokyo, 1984.

\bibitem{AuFr}
J.P. Aubin and H. Frankowska,
{\it Set-Valued Analysis},
Birkh\"auser Boston Inc., Boston, MA, 1990.

\bibitem{A:H}
B. Aulbach and S. Hilger, A unified approach to continuous and
discrete dynamics. Qualitative theory of differential equations,
{\it Colloq. Math. Soc. J\'{a}nos Bolyai}, 53,
North-Holland, Amsterdam, 1990, pp.~37--56.

\bibitem{B:T:08}
Z. Bartosiewicz and D.F.M. Torres,
Noether's theorem on time scales,
{\it J. Math. Anal. Appl.}, vol.~342, no.~2, 2008, pp~1220--1226.
{\tt arXiv:0709.0400}

\bibitem{BeSo}
E.M. Bednarczuk, W. Song,
Contingent epiderivative and its applications to set-valued optimization,
{\em Control Cybernet.}, vol.~27, 1998, pp~375--386.

\bibitem{B:04}
M. Bohner, Calculus of variations on time scales,
{\it Dynam. Systems Appl.}, vol.~13, no.~3-4, 2004, pp~339--349.

\bibitem{MIA}
M. Bohner, R.A.C. Ferreira and D.F.M. Torres,
Integral inequalities and their applications
to the calculus of variations on time scales,
{\it Math. Inequal. Appl.}, vol.~13, no.~3, 2010, pp~511--522.
{\tt arXiv:1001.3762}

\bibitem{B:P:01}
M. Bohner and A. Peterson,
{\it Dynamic equations on time scales.
An introduction with applications},
Birkh\"{a}user Boston, Inc., Boston, MA, 2001.

\bibitem{B:P:03}
M. Bohner and A. Peterson,
{\it Advances in dynamic equations on time scales},
Birkh\"{a}user Boston, Inc., Boston, MA, 2003.

\bibitem{Di}
C.~Dinu, Convex functions on time scales,
{\it Ann. of the Univ. of Craiova, Math. Comp. Sci. Ser.},
vol.~35, 2008, pp~87--96.

\bibitem{F:T:07}
R.A.C. Ferreira and D.F.M. Torres,
Remarks on the calculus of variations on time scales,
{\it Int. J. Ecol. Econ. Stat.}, vol.~9, no.~F07, 2007,
pp~65--73. {\tt arXiv:0706.3152}

\bibitem{F:T:08}
R.A.C. Ferreira and D.F.M. Torres,
Higher-order calculus of variations on time scales,
{\it in Mathematical control theory and finance},
Springer Berlin, 2008, pp~149--159.
{\tt arXiv:0706.3141}

\bibitem{F:T:09}
R.A.C. Ferreira and D.F.M. Torres,
Isoperimetric problems of the calculus of variations on time scales,
Nonlinear Analysis and Optimization II. Optimization, pp.~123--131,
{\it Contemp. Math.}, Amer. Math. Soc. 514, Providence, RI, 2010.
{\tt arXiv:0805.0278}

\bibitem{China}
E. Girejko, A.B. Malinowska and D.F.M. Torres,
A unified approach to the calculus of variations on time scales,
Proceedings 22nd Chinese Control and Decison Conference (2010 CCDC),
Xuzhou, China, May 26-28, 2010. 
In: IEEE Catalog Number CFP1051D-CDR, 2010, 595--600.
{\tt arXiv:1005.4581}

\bibitem{G:M:W:09}
E. Girejko, D. Mozyrska and M. Wyrwas,
Contingent epiderivatives of functions on time scales,
Workshop in Control, Nonsmooth Analysis,
and Optimization, May 4-8, 2009, Porto, Portugal.

\bibitem{SH}
S. Hilger, Analysis on measure chains---a unified approach
to continuous and discrete calculus, {\it Results Math.},
vol.~18, no.~1-2, 1990, pp~18--56.

\bibitem{JaKh}
J. Jahn, A.A. Khan,
Generalized contingent epiderivatives in set-valued optimization:
Optimality conditions, {\em Numer. Funct. Anal. Optim.},
vol.~23, 2002, pp~807--831.

\bibitem{JaRa}
J. Jahn, R. Rauh,
Contingent epiderivatives and set-valued optimization,
{\em Math. Methods Oper. Res.}, vol.~46, 1997, pp~193--211.

\bibitem{Lak:book}
V. Lakshmikantham, S. Sivasundaram and B. Kaymakcalan,
{\it Dynamic systems on measure chains. Mathematics and its Applications}, 370.
Kluwer Academic Publishers Group, Dordrecht, 1996.

\bibitem{AM:T}
A.B. Malinowska and D.F.M. Torres,
Necessary and sufficient conditions for local Pareto optimality on time scales.
{\it J. Math. Sci. (N. Y.)}, vol.~161, no.~6, 2009, pp~803--810.
{\tt arXiv:0801.2123}

\bibitem{AM:T2}
A.B. Malinowska and D.F.M. Torres,
Strong minimizers of the calculus of variations
on time scales and the Weierstrass condition.
{\it Proc. Est. Acad. Sci.},
vol.~58, no.~4, 2009, pp~205--212.
{\tt arXiv:0905.1870}

\bibitem{basia:Leitmann}
A.B. Malinowska and D.F.M. Torres,
Leitmann's direct method of optimization
for absolute extrema of certain problems
of the calculus of variations on time scales,
{\it Appl. Math. Comput.}, 2010,
in press. DOI:10.1016/j.amc.2010.01.015
{\tt arXiv:1001.1455}

\bibitem{Basia:post_doc_Aveiro:2}
A.B. Malinowska and D.F.M. Torres,
Natural boundary conditions in the calculus of variations,
{\it Math. Meth. Appl. Sci.}, 2010, in press.
DOI:10.1002/mma.1289
{\tt arXiv:0812.0705}

\bibitem{AM:T:Poland}
A.B. Malinowska and D.F.M. Torres,
The delta-nabla calculus of variations,
{\it Fasc. Math.}, no.~44, 2010, pp~75--83.
{\tt arXiv:0912.0494}

\bibitem{NM:T}
N. Martins and D.F.M. Torres,
Calculus of variations on time scales with nabla derivatives,
{\it Nonlinear Anal.}, vol.~71, no.~12, 2009, pp~e763--e773.
{\tt arXiv:0807.2596}

\bibitem{comNatalia}
N. Martins and D.F.M. Torres,
Necessary conditions for linear noncooperative
$N$-player delta differential games on time scales,
{\it Discuss. Math. Differ. Incl. Control Optim.},
vol.~30, no.~2, 2010, in press.
{\tt arXiv:0801.0090}

\bibitem{Boris}
B.S. Mordukhovich,
{\it Variational analysis and generalized differentiation. II},
Springer, Berlin, 2006.

\bibitem{RoSa}
L. Rodr\'{\i}guez-Mar\'{\i}n, M. Sama,
About contingent epiderivatives,
{\em J. Math. Anal. Appl.}, vol.~327, 2007, 745--762.

\end{thebibliography}
\end{document}